\documentclass[reqno, 12pt]{amsart}
\usepackage[mathscr]{eucal}
\usepackage[all]{xy}
\usepackage{mathrsfs}
\usepackage{xypic}
\usepackage{amsfonts}
\usepackage{amsmath}
\usepackage{amssymb,enumerate}
\usepackage{latexsym}
\usepackage{tabularx}
\usepackage{graphicx}
\usepackage{pict2e}
\usepackage{tikz}
\usepackage[pagebackref]{hyperref}
\usepackage[text={160mm,240mm},centering]{geometry}
\usepackage{enumitem}
\usepackage{lineno}
\usepackage{color}
\usepackage{fancyhdr}
\usepackage{todonotes}
\usepackage{comment}

\input diagxy
\input xy

\newtheorem{thm}{Theorem}[section]

\newtheorem{cor}[thm]{Corollary}

\newtheorem{lem}[thm]{Lemma}
\newtheorem{prop}[thm]{Proposition}

\theoremstyle{definition}
\newtheorem{defn}[thm]{Definition}
\newtheorem{rem}[thm]{Remark}

\newtheorem{ex}[thm]{Example}
\numberwithin{equation}{section}

\newcommand{\kah}{K\"ahler }
\newcommand{\idd}{i\partial\overline{\partial}}

\newcommand{\pa}{\partial}
\newcommand{\ov}{\overline}
\newcommand{\xu}{\sqrt{-1}}

\newcommand{\cX}{\mathcal{X}}

\newcommand{\dbar}{\bar \partial}
\newcommand{\ddbar}{\partial\bar{\partial}}

\newcommand{\mc}{\mathbb{C}}

\begin{document}


\title{On the direct image of the adjoint big and nef line bundles}

\author{Yuta Watanabe}
\address{Graduate School of Mathematical Sciences, The University of Tokyo, 3-8-1 Komaba, Meguro-Ku, Tokyo 153-8914, Japan}
\email{wyuta.math@gmail.com}

\author{Yongpan Zou}
\address{Graduate School of Mathematical Sciences, The University of Tokyo, 3-8-1 Komaba, Meguro-Ku, Tokyo 153-8914, Japan}
\email{598000204@qq.com}


\date{\today}


\begin{abstract}
We investigate the positivity properties of the direct image $f_{\ast}(K_{X/Y} \otimes L)$ of the adjoint line bundle associated with a big and nef line bundle $L$, under a smooth fibration $f: X\to Y$ between projective varieties. We show that the vector bundle $f_{\ast}(K_{X/Y} \otimes L)$ is big.
\end{abstract}

\maketitle
\setcounter{tocdepth}{1}

\section{Introduction}

For a projective surjective morphism $f: X \to Y$ of complex manifolds with connected fibers, and let $L$ be a line bundle on $X$, we are very interested in the direct image $f_{\ast}(K_{X/Y} \otimes L)$. In general, the positivity of the bundle $L$ induces positivity in the direct image sheaves.

\begin{thm} \cite{Mou97, MT07, Bo09, BLNN23}
Let $f: X\to Y$ be a smooth fibration of smooth projective varieties and denote the relative canonical line bundle by $K_{X/Y}$. For any ample or positive line bundle $L$ on $X$, the direct image $f_{\ast}(K_{X/Y} \otimes L)$ is either zero or an ample vector bundle.
\end{thm}

In \cite{MT07}, Mourougane and Takayama established, through algebro-geometric methods, that $f_{\ast}(K_{X/Y} \otimes L)$ is Griffiths positive as a vector bundle. Furthermore, in \cite{Bo09}, Berndtsson demonstrated that $f_{\ast}(K_{X/Y} \otimes L)$ possesses Nakano positivity as a vector bundle. Considering the big and nef line bundle instead of the ample line bundle is a natural choice. Recently, Biswas, Laytimi, Nagaraj, and Nahm proved the following theorem.
\begin{thm} \cite[Theorem 1.2]{BLNN23} \label{blnn}
Let $f: X\to Y$ be a smooth fibration of smooth projective varieties. If the line bundle $L$ is nef and $f$-strongly big, then $f_{\ast}(K_{X/Y} \otimes L)$ is also nef and big.
\end{thm}
Here a line bundle $L$ on $X$ is said to be $f$-strongly big if there is an effective divisor $D$ on $Y$ with simple normal crossing support such that $L^{d}\otimes f^{*}(\mathcal{O}_Y(-D))$ is ample for some integer $d > 0$. At the end of their paper, they conjecture that the conditions of $L$ being nef and big are sufficient to draw the same conclusion. In this note, we confirm their conjecture and establish the following result.


\begin{thm} [=Theorem \ref{main}] \label{mainthm}
Let $f: X\to Y$ be a smooth fibration of smooth projective varieties. If the line bundle $L$ is big and nef, then $f_{\ast}(K_{X/Y} \otimes L)$ is also nef and $V$-big.
\end{thm}


The Theorem \ref{blnn} above, obtained from algebraic geometry, is generalized by using complex analytic methods here. It is already known that the direct image sheaf $f_{\ast}(K_{X/Y} \otimes L)$ is locally free, i.e., a vector bundle, and nef; see \cite{Mou97} for details. Let $\mathcal{F}:= f_{\ast}(K_{X/Y} \otimes L)$, and denote $\mathcal{O}_{\mathbb{P}(\mathcal{F})}(1)$ as the tautology bundle on $\mathbb{P}(\mathcal{F})$ with respect to $\mathcal{F}$. The nef and big (i.e. $L$-big) properties of $\mathcal{F}$ correspond to the nef and big properties of the line bundle $\mathcal{O}_{\mathbb{P}(\mathcal{F})}(1)$. In this article, we investigate the bigness of $\mathcal{F}= f_{\ast}(K_{X/Y} \otimes L)$. 

As a direct application, combined with the result of M. Iwai in \cite[Theorem $1.2$]{Iwa21}, which enables the extension of vanishing theorems to vector bundles with high-rank multiplier ideal sheaves, we can establish K\"ollar--Ohsawa type vanishing theorems.
\begin{cor} \label{ko-oh}
    Let $f: X\to Y$ be a proper holomorphic surjective morphism of smooth projective varieties with connected fibers. 
    If the holomorphic line bundle $L$ on $X$ is nef and big, then 
we have the following cohomology vanishing
    $$ H^q(Y, f_*(K_X\otimes L)) = 0
    $$
   for any integers $q\geq 1$. Here, $f_{\ast}(K_{X/Y} \otimes L)$ has a induced 
singular Hermitian metric $H$ and there exists a proper analytic subset $Z$ such that $H$ is smooth and Nakano positive on $X\setminus Z$.
\end{cor}

\noindent\textbf{Acknowledgement}: The authors thank their advisor Professor Shigeharu Takayama for his guidance, discussion, and warm encouragement.

\section{Positivity in complex geometry}
In this section, we begin by introducing fundamental definitions and results in complex geometry and algebraic geometry. Unless specified otherwise, $X$ denotes a complex manifold of dimension $n$. Our primary reference for these foundational concepts is \cite{Dem12}.

Let $(E, h)$ be a holomorphic vector bundle of $\mathrm{rank}\,r$ on $X$ with smooth Hermitian metric $h$. Corresponding to the metric $h$, there exists the unique Chern connection $D^h$, which can be split in a unique way as a sum of a $(1,0)$ and a $(0,1)$ connection, i.e., $D^h=D'^h + D''^h$.
Furthermore, the $(0,1)$ part of the Chern connection $D''^h =\dbar$. The curvature form is defined to be $\Theta_{(E,h)} := (D^h)^2$.
On a coordinate patch $\Omega \subset X$ with complex coordinate $(z_1,\cdots,z_n)$, the Chern curvature tensor $\Theta_{(E,h)}$ is written as 
$$
\Theta_{(E,h)}=\sum_{1\leq j,k\leq n}\Theta^h_{jk}dz_j\wedge d\overline{z}_k,
$$
where the coefficients are written as $\Theta^h_{jk}=[D'^h_{z_j},\overline{\partial}_{z_k}]$ and $\overline{\partial}_{z_j}=\partial/\partial \overline{z}_j$.
Putting $(e_1,\cdots,e_r)$ to be an orthonormal frame of $E$ with respect to $h$, we can write 
$$
\Theta_{(E,h)} = \sum_{1\leq j,k\leq n, 1\leq \lambda,\mu \leq r} c_{jk\lambda\mu} dz_j\wedge d\ov{z}_k \otimes e^{\ast}_{\lambda} \otimes e_{\mu} , \quad c_{jk\mu\lambda}=\ov{c}_{jk\lambda\mu}.
$$
When $r=1$, the line bundle case $(E, h=e^{-\phi})$, the formula would be easier, we have $\xu \Theta_{(E,h)} = \xu \ddbar \phi$. Even if $\phi$ is locally defined, it is well-defined.

\begin{defn} [Positive vector bundle]
Let $E$ be a holomorphic vector bundle $E$. We say that a smooth Hermitian metric $h$ on $E$ is 
\begin{itemize}
        \item $\it{Griffiths ~positive}$ if for any $x\in X$, $0\ne\xi\in T_{X,x}\cong\mathbb{C}^n$ and any $0\ne s\in E_x$, we have
        \begin{align*}
            \Theta_{(E,h)}(u\otimes\xi)=\sum(\Theta^h_{jk}u,u)_h\xi_j\overline{\xi}_k>0.
        \end{align*}
        \item $\it{Nakano ~positive}$ if $\Theta_{E,h}$ is positive definite as a Hermitian form on $T_X\otimes E$,
        i.e. for any $x\in X$ and any $0\ne u\in T_X\otimes E$, we have 
        \begin{align*}
            \Theta_{(E,h)}(u)=\sum(\Theta^h_{jk}u_j,u_k)_h>0,
        \end{align*}
        where $u=\sum u_j\otimes e_j\in T_{X,x}\otimes E_x$ for an orthonormal basis $(e_1,\cdots,e_r)$ of $E$.
\end{itemize}
It is clear that Nakano positivity implies Griffiths positivity and that both concepts coincide if $r=1$. In the case of a line bundle, $E$ is merely said to be positive (resp. semi-positive).
\end{defn}

\noindent Next, we consider singular Hermitian metrics and its positivity. We'll start by considering the case of line bundles.

\begin{defn} [Singular metric and curvature current on line bundles] \label{smcc}
Let $(F,h)$ be a holomorphic line bundle on complex manifold $X$ endowed with possible \textit{singular Hermitian metric} $h$. For any given trivialization $\theta : F|_{\Omega} \simeq \Omega \times \mathbb{C}$ by
$$ | \xi |^2_h = |\theta(\xi)|^2 e^{-\phi(x)},  \quad x \in \Omega, \xi \in F_x,
$$
where $\phi \in L^1_{loc}(\Omega)$ is a locally integrable function, called the weight of the metric. 
We say that a singular Hermitian metric $h$ is \textit{singular semi-positive} (resp. \textit{singular positive}) if the weight of $h$ for any trivializations coincides with some (resp. strictly) plurisubharmonic (psh for short) function almost everywhere.
\end{defn}

\noindent The curvature $\sqrt{-1} \Theta_{(F,h)}$ of $h$ is defined by
$$  \sqrt{-1} \Theta_{(F,h)} = \sqrt{-1} \partial \overline{\partial} \phi.
$$
The Levi form $\sqrt{-1}\pa\dbar \phi$ is taken in the sense of distributions and thus the curvature is a $(1,1)$-current but not always a smooth $(1,1)$-form. It is globally defined on $X$ and independent of the choice of trivializations.


\begin{defn} [Multiplier ideal sheaves]
Let $\phi$ be a quasi-plurisubharmonic (quasi-psh) function on a complex manifold $X$; that is, locally, it is the sum of a psh function and a smooth function, the \textit{multiplier ideal sheaf} $\mathcal{I}(\phi) \subset \mathcal{O}_X$ is defined by
$$ \Gamma(U, \mathcal{I}(\phi)) = \{f \in \mathcal{O}_X(U) :\quad |f|^2e^{-\phi} \in L^1_{loc}(U) \}
$$
for every open set $U \subset X$. For a line bundle $(F, h)$, with the local weight of the metric $h$ denoted by $\phi$, we interchangeably denote the multiplier ideal sheaf as $\mathcal{I}(\phi)$ or $\mathcal{I}(h)$.
\end{defn}

\noindent We already know the following characterization of the algebraic geometric positivity of line bundles

\begin{thm}\cite[Chapter $6$]{Dem10}\label{metricdes}
Let $X$ be a projective manifold equipped with a \kah metric $\omega$. Then we have that a holomorphic line bundle $L$ is
    \begin{itemize}
        \item nef if and only if for any $\varepsilon>0$ there exists a smooth Hermitian metric $h_\varepsilon$ on $L$ such that $i\Theta_{(L,h_\varepsilon)}\geq-\varepsilon\omega$,
        \item big if and only if there exists a singular Hermitian metric $h$ on $L$ such that $i\Theta_{(L,h)}\geq\varepsilon\omega$,
        for some $\varepsilon>0$.
    \end{itemize}
\end{thm}

Clearly, singular semi-positivity and singular positivity coincide with pseudo-effective and big, respectively, on projective manifolds.

\begin{lem} [Kodaira lemma] \label{kod}
    Let $L$ be a big line bundle, for any ample integer divisor $A$ on $X$, there exists a positive integer $m>0$ and an effective divisor $E$ on $X$ such that $m\cdot L = A + E$.
\end{lem}

Given a vector bundle $E$ on $X$, $S^m E$ is the $m$-th symmetric product of $E$, and
$$ \pi: \mathbb{P}(E) \to X
$$
denotes the projective bundle of one-dimensional quotients of $E$. As usual, $\mathcal{O}_{\mathbb{P}(E)}(1)$ is the Serre line bundle on $\mathbb{P}(E)$, i.e., the tautological quotient of $\pi^{\ast} E$, and we have $\textnormal{Sym}^m E = \pi_{\ast}\mathcal{O}_{\mathbb{P}(E)}(m)$.


We introduce the notions of ample, nef and big for vector bundles. 
Let $X$ be a projective variety and $E$ be a holomorphic vector bundle on $X$.

\begin{defn}
    \cite[Section\,2]{BKK+15}
    We define the \textit{base locus} of $E$ as the set 
    \begin{align*}
        \textnormal{Bs}(E):=\{x\in X\mid H^0(X,E)\to E_x\,\, \textnormal{is not surjective}\},
    \end{align*}
    and the \textit{stable base locus} of $E$ as 
    \begin{align*}
        \mathbb{B}(E):=\bigcap_{m>0}\textnormal{Bs}(\textnormal{Sym}^mE).
    \end{align*}
    For an ample line bundle $A$, we also define the \textit{augmented base locus} of $E$ as 
    \begin{align*}
        \mathbb{B}^A_+(E):=\bigcap_{r\in\mathbb{Q}_{>0}}\mathbb{B}(E-rA).
    \end{align*}
    Here, $\mathbb{B}^A_+(E)$ dose not depend on the choice of the ample line bundle. 
    Therefore, we write $\mathbb{B}_+(E)$ for simplicity.
\end{defn}

With these notations, we introduce the following definitions.

\begin{defn}
    \cite[Theorem\,1.1,\,Definition\,5.1,\,Definition\,6.1]{BKK+15}
    Let $X$ be a smooth projective variety and $E$ be a vector bundle on $X$. 
    We say that 
    \begin{itemize}
        \item $E$ is ample if $\mathcal{O}_{\mathbb{P}(E)}(1)$ is ample on $\mathbb{P}(E)$,
        \item $E$ is nef if $\mathcal{O}_{\mathbb{P}(E)}(1)$ is nef on $\mathbb{P}(E)$,
        \item $E$ is $L$-\textit{big} if $\mathcal{O}_{\mathbb{P}(E)}(1)$ is big on $\mathbb{P}(E)$,
        \item $E$ is $V$-\textit{big} (or \textit{Viehweg-big}) if $\mathbb{B}_+(E)\ne X$.
    \end{itemize}
\end{defn}

Note that if $E$ is $V$-big, then $E$ is $L$-big (see. \cite[Corollary\,6.5]{BKK+15}) and that 
\begin{align*}
    \pi(\mathbb{B}_+(\mathcal{O}_{\mathbb{P}(E)}(1)))=\mathbb{B}_+(E) \qquad (\textnormal{see \cite[Proposition\,3.2]{BKK+15}}).
\end{align*}

Moreover, if the vector bundle $E$ is positive in the sense of Griffiths, then the line bundle $\mathcal{O}_{\mathbb{P}(E^{})}(1)$ is a positive line bundle, the reader can find the curvature formula in \cite[Chapter V formula $15.15$]{Dem12}.


We now introduce positivity notions for singular Hermitian metrics for vector bundles.
Let $H_r$ be the space of semi-positive, possibly unbounded Hermitian forms on $\mathbb{C}^r$. A singular Hermitian metric $h$ on vector bundle $E$ is a measurable map from $X$ to $H_r$ such that $h(x)$ is finite and positive definite almost everywhere. In particular, we require $0< \det h < +\infty$ almost everywhere.

\begin{defn}\cite{BP09, PT}\label{def Grif semi-posi as sing}
Let $(E,h)$ be a singular Hermitian metric on $X$, then $(E,h)$ is said to be:
\begin{itemize}
		\item \textit{Griffiths semi-negative} if $\log |s|^2_h$ (or $|s|^2_h$) is psh for any local holomorphic section $s$ of $E$.
		\item \textit{Griffiths semi-positive} if the dual metric $h^\star$ on $E^\star$ is Griffiths semi-negative.
	\end{itemize}
\end{defn}

\begin{lem}  \cite{BP09, Rau15} \label{regular}
In the case of singular metrics, we can construct an approximation sequence of smooth metrics using convolution.
\begin{itemize}
 \item Suppose $X$ is a polydisc in $\mathbb{C}^n$, and suppose $h$ is a singular Hermitian metric on $E$ which is Griffiths semi-negative (resp. Griffiths semi-positive). Then, on any smaller polydisc, there exists a sequence of smooth Hermitian metrics $\{h_\nu\}$ decreasing (resp. increasing) pointwise to $h$ whose corresponding curvature tensor is Griffiths semi-negative (resp. semi-positive).
\end{itemize}
\end{lem}

The following lemma tells us that the quotient and pull-back preserve semi-positivity.

\begin{lem} \cite[Lemma $2.2.2$, $2.3.4$]{PT} \label{quo}
\begin{enumerate}
\item Let $h$ be a singular Hermitian metric on $E$ and $E\to Q$ be a quotient vector bundle. Suppose that $h$ is Griffiths semi-positive. Then $Q$ has a naturally induced singular Hermitian metric $h_Q$ with Griffiths semi-positivity.
\item  Let $f: X\to Y$ be a holomorphic surjective map between two complex manifolds and let $E$ be a vector bundle on $Y$. If $E$ admits a singular Hermitian metric $h$ with Griffiths semi-positivity. Then $f^*h$ is a singular Hermitian metric on $f^*E$, and it is Griffiths semi-positive.
\end{enumerate}
\end{lem}

\section{Positivity of induced metrics on the tautological line bundle} 

In this section, we introduce the definition of singular Griffiths positivity and demonstrate that this definition induces singular positivity for tautological line bundles.

\begin{defn}\label{def of strictly Grif posi as sing}
    Let $X$ be a complex manifold and $E$ be a holomorphic vector bundle on $X$. 
    We say that a singular Hermitian metric $h$ on $E$ is 
    \begin{itemize}
        \item \textit{Griffiths negative} if for any $x\in X$, there exist an open neighborhood $U$ of $x$ and $\delta>0$ such that for any local holomorphic section $u\in H^0(U,E)$, we have 
        \begin{align*}
            \idd|u|^2_h\geq\delta|u|^2_h\cdot\idd|z|^2
        \end{align*}
        on $U$ in the sense of currents, where $(z_1,\cdots,z_n)$ is local coordinates of $U$,
        \item \textit{Griffiths positive} if the dual metric $h^*$ on $E^*$ is Griffiths negative.
    \end{itemize}
\end{defn}

There is a definition close to Definition \ref{def of strictly Grif posi as sing} (cf. \cite[Definition\,6.1]{Rau15}).
In this paper, we adopt Definition \ref{def of strictly Grif posi as sing} as a more general setting and show that this definition induces singular positivity for the tautological line bundle.
Here, this inequality $\idd|u|^2_h\geq\delta|u|^2_h\cdot\idd|z|^2$ is equivalent to Griffiths negative if $h$ is smooth. Similar to the proofs in \cite{Rau15}, we obtain the following.

\begin{prop}\label{exist smoothing of Grif posi}
    Let $X$ be a complex manifold and $E$ be a holomorphic vector bundle on $X$ equipped with a singular Hermitian metric $h$. 
    If $h$ is Griffiths negative then for any Stein subset $S$ which satisfies $E|_S$ is trivial, 
        there exists a sequence of smooth Griffiths negative Hermitian metrics $(h_\nu)_{\nu\in\mathbb{N}}$ on any relatively compact Stein subset $\widetilde{S}$ of $S$ where $h_\nu:=h\ast\rho_\nu$ and $(\rho_\nu)_{\nu\in\mathbb{N}}$ is an approximate identity on $S$.
        And this sequence decreasing to $h$ a.e. pointwise and satisfies 
        that there exists $\delta>0$ such that 
        \begin{align*}
            \quad \idd|u|^2_{h_\nu}\geq\delta|u|^2_{h_\nu}\cdot\idd|z|^2,  
        \end{align*}
        for any $\nu\in\mathbb{N}$, any $x\in\widetilde{S}$ and any $0\ne u\in E_x$.
\end{prop}

Let $\pi$ be the projection $\mathbb{P}(E)\to X$ and $L:=\mathcal{O}_E(1)$.
Any smooth Hermitian metric $h$ on $E$ introduces a \textit{canonical metric} $h_L$ on $L$ as follows:

Let $e=(e_1,\ldots,e_r)$ be a local holomorphic frame with respect to a given trivialization on $E$.
The corresponding holomorphic coordinates on $E^*$ are denoted by $(w_1,\ldots,w_r)$. 
There is a local section of $e_{L^*}$ of $L^*$ defined by 
\begin{align*}
    e_{L^*}=\sum_{1\leq\lambda\leq r}w_\lambda e^*_\lambda.
\end{align*}
Define the induced quotient metric $h_L$ on $L$ by the morphism $(\pi^*E,\pi^*h)\to L$. 
In other words, the dual metric $h_L^*$ on $L^*$ is a metric induced by the natural mapping $\iota: L^* \hookrightarrow (\pi^*E)^*=\pi^*E^*$.
If $(h_{\lambda\mu})$ is the matrix representation of $h$ with respect to the basis $e$, then $h_L$ can be written as 
\begin{align*}
    h_L=\frac{1}{||e_{L^*}||^2_{h_{L^*}}}=\frac{1}{||\iota\circ e_{L^*}||^2_{\pi^*h^*}}=\frac{1}{||e_{L^*}||^2_{h^*}}=\frac{1}{\sum h^*_{\lambda\mu}w_\lambda\overline{w}_\mu},
\end{align*}
where $h^*_{\lambda\mu}=h^{\mu\lambda}$ and $(h^{\lambda\mu}),(h^*_{\lambda\mu})$ are the matrix representation of $h^{-1}, h^*$, respectively. 
The curvature of $(L,h_L)$ is 
    \begin{align*}
        i\Theta_{(L,h_L)}=-\idd\log h_L=\idd\log||e_{L^*}||^2_{h^*}=\idd\log\Bigl(\sum h^{\mu\lambda}w_\lambda\overline{w}_\mu\Bigr).
    \end{align*}
We fix a point $t\in X$. 
    By \cite[ChapterV]{Dem12}, 
    for any standard coordinate $z=(z_1,\ldots,z_n)$ of $\omega$ centered at point $t$, i.e. $\omega=i\sum dz_j\wedge d\overline{z}_j+O(|z|)$,
    there exists a local holomorphic frame $e=(e_1,\ldots,e_r)$ of $E$ around $t$ such that 
    \begin{align*}
        h_{\lambda\mu}=\delta_{\lambda\mu}-\sum_{j,k}c_{jk\lambda\mu}z_j\overline{z}_k+O(|z|^3),
    \end{align*}
    where these coefficients $c_{jk\lambda\mu}$ are of curvature $\Theta_{(E,h)}$.
    From the following inequality 
    \begin{align*}
        i\Theta_{(E^*,h^*)}(t)&=i\sum c^*_{jk\lambda\mu}dz_j\wedge d\overline{z}_k\otimes (e^*_\lambda)^*\otimes e^*_\mu \\
        =-i\,{}^t\Theta_{(E,h)}&=-i\sum c_{jk\mu\lambda}dz_j\wedge d\overline{z}_k\otimes (e^*_\lambda)^*\otimes e^*_\mu, 
    \end{align*}
    i.e. $c^*_{jk\lambda\mu}=-c_{jk\mu\lambda}$, we have that 
    \begin{align*}
        h^{\mu\lambda}=h^*_{\lambda\mu}=\delta_{\lambda\mu}-\sum c^*_{jk\lambda\mu}z_j\overline{z}_k+O(|z|^3)
        =\delta_{\lambda\mu}+\sum c_{jk\mu\lambda}z_j\overline{z}_k+O(|z|^3).
    \end{align*}
 Therefore, for any point $w\in E_t$, i.e. any point $p=(t,[w])\in\mathbb{P}(E)$ where $[w]=[w_1:\cdots:w_r]$ and $\pi(p)=t$, we have that 
    \begin{align*}
        i\Theta_{(L,h_L)}(p)&=i\sum c_{jk\mu\lambda}\frac{w_\lambda\overline{w}_\mu}{|w|^2}dz_j\wedge d\overline{z}_k+i\sum_{1\leq \lambda,\mu\leq r}\frac{|w|^2\delta_{\lambda\mu}-\overline{w}_\lambda w_\mu}{|w|^4}dw_\lambda\wedge d\overline{w}_\mu\\
        &=i\sum c_{jk\lambda\mu}\frac{w_\mu\overline{w}_\lambda}{|\overline{w}|^2}dz_j\wedge d\overline{z}_k+\omega_{FS}([w])\\
        &=i\frac{\Theta_{(E,h)}(\overline{w})}{|\overline{w}|^2_h}+\omega_{FS}([w]),
    \end{align*}
    where $\omega_{FS}$ is Fubini-Study metric on $\mathbb{P}(E^*_t)\cong\mathbb{P}^{r-1}$.

\begin{thm}\label{E Grif posi then O_E(1) is also big}
    Let $X$ be a complex manifold and $E$ be a holomorphic vector bundle on $X$ equipped with a singular Hermitian metric $h$.
    If $h$ is Griffiths positive then the tautological line bundle $\mathcal{O}_E(1)\to\mathbb{P}(E)$ has a naturally induced singular Hermitian metric and this metric is singular positive.
    In particular, if $X$ is compact then $\mathcal{O}_E(1)$ is big.
\end{thm}

\begin{proof}
    From Proposition \ref{exist smoothing of Grif posi}, for any $x\in X$, there exists a neighborhood $U$ of $x$, $\delta_x>0$ and a sequence of smooth Griffiths negative Hermitian metrics $(h^*_\nu)_{\nu\in\mathbb{N}}$ on $U$ decreasing to $h^*$ a.e. pointwise such that 
    \begin{align*}
        \quad \idd|u|^2_{h^*_\nu}\geq\delta_x|u|^2_{h^*_\nu}\idd|z|^2,  
    \end{align*}
    for any $\nu\in\mathbb{N}$, any $y\in U$ and any $0\ne u\in E^*_y$. 
    By the inequality
    \begin{align*}
        \idd|v|^2_{h^*_\nu}=\sum_{j,k}(D'^{h^*_\nu}_{z_j}v,D'^{h^*_\nu}_{z_k}v)_{h^*_\nu}dz_j\wedge d\overline{z}_k-\sum_{j,k}(\Theta^{h^*_\nu}_{jk}v,v)_{h^*_\nu}dz_j\wedge d\overline{z}_k,
    \end{align*}
    for any local holomorphic section $v\in\mathcal{O}(E^*)_y$, we have that 
    \begin{align*}
        -\sum_{j,k}(\Theta^{h^*_\nu}_{jk}u,u)_{h^*_\nu}dz_j\wedge d\overline{z}_k\geq\delta_x|u|^2_{h^*_\nu}\idd|z|^2.
    \end{align*}
    In fact, for any $u\in E^*_y$ we can take $v\in\mathcal{O}(E^*)_y$ satisfying $v(y)=u$ and $D'^{h^*_\nu}v(y)=0$.
    
 There is a natural anti-linear isometry between $E^*$ and $E$, which we will denote by $J_\nu$.
    Denote the pairing between $E^*$ and $E$ by $\langle\cdot,\cdot\rangle$ which satisfies that $\langle\xi,u\rangle=(u,J_\nu\xi)_{h_\nu}$
    for any local section $u$ of $E$ and any local section $\xi$ of $E^*$.
    Under the natural holomorphic structure on $E^*$, we obtain
    \begin{align*}
        \overline{\partial}_{z_j}\xi=J^{-1}_\nu D'^{h_\nu}_{z_j}J_\nu\xi, \quad D'^{h^*_\nu}_{z_j}\xi=J^{-1}_\nu\overline{\partial}_{z_j}J_\nu\xi.
    \end{align*}
    Thus, for any local sections $\xi_j\in C^\infty(E^*)$ and $u_j\in C^\infty(E)$ satisfying $u_j=J_\nu\xi_j$, we get 
    \begin{align*}
        \sum(\Theta^{h^*_\nu}_{jk}\xi_j,\xi_k)_{h^*_\nu}=-\sum(\Theta^{h_\nu}_{jk}u_k,u_j)_{h_\nu},
    \end{align*}
 Hence, for any $\nu\in\mathbb{N}$, any $y\in U$ and any $0\ne \xi\in E_y$, we have that 
    \begin{align*}
        \Theta_{(E,h_\nu)}(\xi)=\sum_{j,k}(\Theta^{h_\nu}_{jk}\xi,\xi)_{h_\nu}dz_j\wedge d\overline{z}_k\geq\delta_x|\xi|^2_{h_\nu}\idd|z|^2.
    \end{align*}
Let $h_L^\nu$ be canonical metrics on $L|_{\pi^{-1}(U)}$ induced by $h_\nu$ where 
    $L:=\mathcal{O}_E(1)$.
    Let $(e_1,\ldots,e_r)$ be a orthonormal basis on $E$, then we can write 
    \begin{align*}
        i\Theta_{(E,h_\nu)}&=i\sum c^\nu_{jk\lambda\mu}dz_j\wedge d\overline{z}_k\otimes e^*_\lambda\otimes e_\mu
    \end{align*}
    at $y\in U$.
    Then for any point $p=(y,[w])\in\mathbb{P}(E)$, i.e. any $w\in E_y$, we have 
    \begin{align*}
        i\Theta_{(L,h_L^\nu)}(p)&=i\sum c^\nu_{jk\lambda\mu}\frac{w_\mu\overline{w}_\lambda}{|\overline{w}|^2_{h_\nu}}+\omega_{FS}([w])=i\frac{\Theta_{(E,h_\nu)}(\overline{w})}{|\overline{w}|^2_{h_\nu}}+\omega_{FS}([w])\\
        &\geq \delta_x\idd|z|^2+\omega_{FS}([w]),
    \end{align*}
    i.e. $i\Theta_{(L,h^\nu_L)}\geq\delta_x\idd|z|^2+\widetilde{\omega}_{FS}$ on $\pi^{-1}(U)$, where
    there is a global metric $\widetilde{\omega}_{FS}$ on $\mathbb{P}(E)$ that is $\overline{\partial}$-closed and is the Fubini-Study metric when restricted to each fiber. 
    
    Here, $\delta_x\idd|z|^2+\widetilde{\omega}_{FS}$ is K\"ahler, and since the weights $\varphi_\nu$ of $h^\nu_L$ has a uniform positivity and is strictly plurisubharmonic and decreasing to the weight $\varphi$ of the canonical metric $h_L$ induced by $h$ a.e, then $\varphi$ coincides with some strictly plurisubharmonic function.
    Therefore, the singular Hermitian metric $h_L$ on $L$ is singular positive.
\end{proof}

\section{The bigness of direct image}
For a smooth fibration $f: X \to Y$ between smooth projective varieties, and considering $L$ as a big and nef line bundle on $X$, we investigate the positivity of the direct image $f_{\ast}(K_{X/Y} \otimes L)$, occasionally denoted as $f_{\ast}(K_{X/Y} + L)$.

\subsection{Set-up of the metric on $L$} \label{setup1}
Let $L$ be a big and nef line bundle on $X$, and $\omega$ a Kähler form on $X$. By Kodaira's Lemma \ref{kod}, there exists an integer $m$ such that $m\cdot L = A + N$, where $A$ is an ample divisor and $N$ is an effective divisor. Choosing a positive integer $d$, we obtain:
\begin{equation*}
     md\cdot L =  A + N + (d-1)\cdot (A+N).
\end{equation*}
We can find a positive metric $h_A = e^{-\phi_A}$ on $A$, and a nef metric $h_{\varepsilon} = e^{-\phi_{\varepsilon}}$ on $(d-1)\cdot (A+N)$. Since $N$ is effective, its canonical section induces a singular metric $h_N = e^{-\phi_N}$ which has semi-positive curvature in the sense of current. Therefore we have a metric $h_L$ on $L$ as following:
\begin{equation} \label{metric}
     h_L := e^{-\phi_L} := e^{-\frac{1}{md}\phi_A  - \frac{1}{md}\phi_{N} - \frac{d-1}{md}\phi_{\varepsilon}}.
\end{equation}
 Now, choose a sufficiently large integer $d$ such that the multiplier ideal sheaf $\mathcal{I}(\frac{1}{m d}\phi_N) = \mathcal{O}_X$. Next, select a nef metric $h_{\varepsilon} = e^{-\phi_{\varepsilon}}$ on $(d-1) \cdot (A+N)$ such that $\xu \ddbar\phi_A + \xu \pa\dbar(\phi_{\varepsilon}) \geq \varepsilon \omega$. Consequently, we can ensure that the curvature current of $\xu \ddbar \phi_L \geq \delta \omega$ for some positive number $\delta$, thereby enforcing the metric $h_L$ to satisfy $\mathcal{I}(h_L) = \mathcal{O}_X$.

Now the direct image sheaves
$$ \mathcal F:=f_*((K_{X/Y}+L)\otimes \mathcal I(h_L)) = f_*(K_{X/Y}+L)
$$
endowed with the induced $L^2$-metric $h_{\mathcal F}$ from $h_L$, i.e., for any $u_t \in \mathcal F_t$,
$$ \|u_t\|_{h_{\mathcal F}}^2 = \int_{X_t} c_n u\wedge \bar u e^{-\phi_L} ~~\text{with} ~~ c_n=(\sqrt{-1})^{n^2}. $$
We have the next classical results:
\begin{thm} \cite{PT,HPS18} \label{semi-po}
Let $f: X\to Y$ be the projective surjective morphism between smooth varieties with connected fibers. Let the line bundle $(L,h_L)$ be as in the above set-up \eqref{metric}. Then, the induced metric $h_{\mathcal F}$ on $f_{\ast}((K_{X/Y}+L)\otimes \mathcal I(h_L)) = f_{\ast}(K_{X/Y}+L)$ is Griffiths semi-positive.
\end{thm}

By Lemma \ref{quo}, we establish that the induced metric $h_1$ on the tautology line bundle $\mathcal{O}_{\mathbb{P}(\mathcal{F})}(1)$ is semi-positive. The challenging aspect lies in proving the positivity of $h_1$ on $\mathcal{O}_{\mathbb{P}(\mathcal{F})}(1)$. Referring to the construction of $h_L$ in \eqref{metric}, the singular part is $\phi_N$, which has analytic singularities.

The effective divisor $N$ can be pretty complicated, but we can begin with the simple case, we first assume that $N$ has simple normal crossing support (SNC), i.e., $N= \sum a_i D_i$ for smooth reduce divisor $D_i$ and $a_i \geq 0$. Moreover, we can assume that each component of $N$ intersects with each fiber transversally. In this case, even the metric $h_L$ is singular, but the induced metric $h_{\mathcal{F}}$ on $\mathcal{F}$ is smooth.

Indeed, locally let $\Omega\subset X$ be a coordinate patch on $X$, we take $(z_1, \cdots, z_n, t_1, \cdots, t_m)$ a coordinate system on $X$ such that $f: (z_1, \cdots, z_n, t_1, \cdots, t_m) \rightarrow (t_1, \cdots, t_m)$. If $D_i = \{z_i=0\}$, We can express the local weight $\phi_L$ as $\phi_L = \sum_{i} \frac{a_i}{md}\log |z_i|^2$ modulo smooth functions. One can use partitions of unity to reduce to checking that integrals of the form $\int_{\Omega\cap X_t}|u_t|^2 e^{-\phi_L}$ vary smoothly with $t$. In the next section, we will show $(\mathcal F, h_{\mathcal F})$ is positive in the sense of Nakano, which yields $\mathcal{O}_{\mathbb{P}(\mathcal{F})}(1)$ is positive line bundle.

In the general case, there exist points $t \in Y$ where the divisor $N$ does not intersect the fiber $X_t$ transversally, or $N$ contains some components of $X_t$. Additionally, the divisor $N$ may not have SNC support. We can consider a log resolution $\pi: X' \to X$ of $(X, N, f)$, here we consider $(N, f)$ as the ideal sheaves. Set $f':= f\circ\pi: X' \to Y$, and $(L', h_{L'}) := (\pi^{\ast}L, \pi^{\ast}h_L)$. Let $M$ be the reduced divisor induced by $\pi^{-1}(N)$. We know
$$
\pi_*((K_{X'/Y}+L')\otimes \mathcal I(h_{L'})) = (K_{X/Y}+L)\otimes \mathcal I(h_L).
$$
The map $f'$ may not be smooth anymore, some components of $M$ may be irreducible components of fibers of $f'$.

Let $Y_0$ be the Zariski open set of regular values of $f'$, and take a Zariski open subset $Y_1$ of $Y_0$ such that $M$ intersects each fiber transversally. Set $X_1 := f'^{-1}(Y_1), f_1 := f'|_{X_1}$, and $(L_1, h_{L_1}) := (L', h_{L'})|_{X_1}$. We will focus on $(f_1, L_1, h_{L_1})$. In the next section, we will show $(\mathcal F, h_{\mathcal F})$ is positive in the sense of Nakano on $Y_1$ by $L^2$-method, see Theorem \ref{Thm2}. Indeed, the case $(f_1: X_1\to Y_1, L_1, h_{L_1})$ satisfies the conditions outlined in Set-up \ref{setup2}, particularly exhibiting klt-type singularities as illustrated in Example \ref{3-ex}. Hence we come to the following conclusion. 

\begin{lem} \label{loc}
Let $f: X\to Y$ and the line bundle $(L,h_L)$ be as in the above Set-up \eqref{metric}. Then, there exists a Zariski open subset $U$ of $Y$ such that the induced metric $(h_1, \mathcal{O}_{\mathbb{P}(\mathcal{F})}(1))$ comes from $(h_{\mathcal F}, f_{\ast}(K_{X/Y}+L))$ is smooth and positive on $\pi^{-1}(U)$.
\end{lem}

Now we will prove the main theorem of this paper. We denote the line bundle $\mathcal{O}_{\mathbb{P}(\mathcal{F})}(1)$ by $F$ with the induced metric $h_1$. We want to show that $F$ is big. Now we know that $F$ is nef, and has singular Hermitian metric $h_1$ which is singular semi-positive as in Definition \ref{smcc}, i.e., it is pseudo-effective. Moreover, it is smooth and positive on a Zariski open subset $U$ of $Y$.

Let us recall the algebraic definition of big line bundle on compact K\"ahler manifold. First, we introduce the concept of \emph{Kodaira--Iitaka dimension} of a line bundle. If $F$ is a line bundle, the Kodaira--Iitaka dimension $\kappa(F)$ is the supermum of the rank of the canonical maps:
\begin{align*}
    \Phi_m : &Y\setminus B_m \to \mathbb{P}(V_m)\\
    &x \mapsto H_x:= \{\sigma\in V_m: \sigma(x)=0\}, 
\end{align*}
with $V_m = H^0(Y, mF)$ and $B_m = \cap_{\sigma\in V_m}\sigma^{-1}(0)$ is the base locus of $V_m, m\geq 1$. A line bundle is said to be big if $\kappa(F)= \dim Y$.

The direct method to prove the bigness of $F$ is to show that $\Phi_m$ is birational onto its image. For any point $p\in U$, on the small neighbor $U_p \subset U$, the metric $h_1$ is smooth and positive. For some integer $m_0>0$, we can produce a singular Hermitian metric with a given logarithmic pole $h^{m_0}_1\cdot e^{-\tau(z)\log|z-p|^2}$ in neighbor $U_p$, here $\tau$ is a smooth cut-off function supported on $U_p$. Then Nadel--Nakano vanishing with multiplier ideal sheaf (cf. \cite[Theorem 1.3]{Iwa21}) can be used to produce sections of $L^k$ which generate all jets of order $(\frac{k}{m_0}-n)$ at the points $p$ so that $L$ is big.
\begin{thm} \label{main}
Let $f: X\to Y$ be a smooth fibration of smooth projective varieties. If the line bundle $L$ is big and nef, then $f_{\ast}(K_{X/Y} \otimes L)$ is also nef and $V$-big.
\end{thm}

\begin{proof}
Since we already know the line bundle $F= \mathcal{O}_{\mathbb{P}(\mathcal{F})}(1)$ is nef. According to the classical results, see \cite[Corollary 6.19]{Dem10} for details, a nef line bundle $F$ on $Y$ is big if and only if its top self-intersection $F^m = \int_Y c_1(F)^m >0$.

We know that $c_1(F)$ and $c_1(F,h_1)$ are in the same cohomology class. However, here $(F, h_1)$ is pseudo-effective, the wedge product of current $c_1(F,h_1) = \frac{\sqrt{-1}}{2\pi} \Theta_{(F, h_1)}$ may not be well-defined. Hence, we choose its absolute continuous part $c_1(F,h_1)_{ac}=: \omega_{ac}$, the Radon--Nikodym theorem ensures that $\omega_{ac}$ is a $(1,1)$-form with $L^1_{loc}$ coefficients. Then $\omega_{ac}^m$ exists for almost all $y\in Y$.
Indeed, for a $(1,1)$-current $T$ on $Y$. The Lebesgue decomposition of $T$ is $T= T_{ac} + T_{sing}$, consisting of the absolutely continuous and singular components. If $T \geq 0$, it follows that $T_{ac}\geq 0$.

We introduce an important concept, the so-called volume of the holomorphic line bundle of $F$:
$$
\mathrm{Vol}(F) := \limsup_{p\rightarrow \infty} \frac{m!}{p^m} \mathrm{\dim} H^0(Y, F^p).
$$
Thus, $F$ is big if and only if $\mathrm{Vol}(F)>0$. If moreover, $F$ is smooth and positive, by the Kodaira vanishing and asymptotic Riemann-Roch-Hirzebruch formula, we have $\mathrm{Vol}(F)= \int_Y c_1(F)^m$.

Now by \cite[Lemma 2.3.44]{MM07}, we have $\mathrm{Vol}(F)\geq \int_Y \omega_{ac}^m$. In fact, a theorem of S. Boucksom \cite[Theorem $1.2$]{Bou02} says that 
$$\mathrm{Vol}(F) = \sup\{\int_Y \omega_{ac}^m : \omega\in c_1(F)~ \text{closed semi-positive current}\}$$
and the supremum involved is always finite. By Lemma \ref{loc}, we know that $c_1(F,h_1) = \frac{\sqrt{-1}}{2\pi} \Theta_{(F, h_1)} > 0$ as a smooth $(1,1)$-form on open subset $U$ of $Y$. Hence we have
$$
\mathrm{Vol}(F) \geq  \int_Y \omega_{ac}^m >0.
$$
We conclude that $\mathcal{O}_{\mathbb{P}(\mathcal{F})}(1)$ is big, i.e. $\mathcal{F}$ is $L$-big. 
Here, $\mathbb{B}_+(\mathcal{O}_{\mathbb{P}(\mathcal{F})}(1))\ne\mathbb{P}(\mathcal{F})$ by \cite[Propositon-Definition\,4.2]{BKK+15}.

From Lemma \ref{loc}, there exists a Zariski open subset $U$ of $Y$ such that the induced metric of $\mathcal{O}_{\mathbb{\mathcal{F}}}(1)$ is smooth and positive on $\pi^{-1}(U)$.
For an ample line bundle $A$ on $\mathbb{B}(\mathcal{F})$ and an enough small $\varepsilon\in\mathbb{Q}_{>0}$, we get $\textnormal{Bs}(\mathcal{O}_{\mathbb{\mathcal{F}}}(1)-\varepsilon A)\subset \pi^{-1}(Y\setminus U)=\mathbb{P}(\mathcal{F})\setminus\pi^{-1}(U)$ by the Nadel vanishing.
Hence, we have that 
\begin{align*}
    \mathbb{B}_+(\mathcal{F})&=\pi(\mathbb{B}_+(\mathcal{O}_{\mathbb{\mathcal{F}}}(1)))\subset\pi(\mathbb{B}(\mathcal{O}_{\mathbb{\mathcal{F}}}(1)-\varepsilon A))=\pi\Bigl(\bigcap_{m>0}\textnormal{Bs}(\mathcal{O}_{\mathbb{P}(\mathcal{F})}(m)-\varepsilon A^{\otimes m})\Bigr)\\
    &\subset \pi(\textnormal{Bs}(\mathcal{O}_{\mathbb{P}(\mathcal{F})}(1)-\varepsilon A))\subset \pi(\pi^{-1}(Y\setminus U))=Y\setminus U\ne Y,
\end{align*}
by \cite[Proposition\,3.2]{BKK+15}. Therefore, we know that $\mathcal{F}$ is $V$-big. 
\end{proof}

\begin{rem}
We can also establish the $L$-bigness of $\mathcal{O}_{\mathbb{P}(\mathcal{F})}(1)$ using Theorem \ref{E Grif posi then O_E(1) is also big}, provided that the effective divisor in Set-up \eqref{metric} has simple normal crossing supports.
\end{rem}


\section{Nakano positivity via the optimal $L^2$-estimate}

In \cite{DNWZ20}, the authors give a nice characterization of the Nakano positivity of holomorphic vector bundle with smooth metric via the optimal $L^2$-estimate condition. It is worth noting that in many situations, the metric $h_L$ of the line bundle is singular, but it is smooth concerning the $t$ variables. This means the direct image vector bundle is smooth. Therefore it is possible to use the optimal $L^2$-estimate condition to study the Nakano positivity in this kind of singular case.

\subsection{Set-up of case $Z$} \label{setup2}
Let $p:\cX\to D$ be a holomorphic proper fibration (i.e., submersion) from a $(n+m)$-dimensional K\"ahler manifold $(\cX, \omega)$ onto the bounded pseudoconvex domain $D\subset \mathbb C^{m}$, and let $(L,h_L)$ be a holomorphic line bundle endowed with a possibly singular hermitian metric $h_L$. Let $\Omega\subset \cX$ be a coordinate patch on $\cX$. We take $(z_1,\dots z_n, t _1,\dots, t_m)$ a coordinate system on $\Omega$ such that the last $m$ variables $t_1,\dots, t_m$ corresponds to the map $p$ itself. We assume that $\cX$ can be covered by the system of such a coordinate subset.
\begin{enumerate} [label={\color{violet}(Z.\arabic*)}]
\item \label{Z1} The metric $h_L = e^{-\psi_L}$ and the local weights $\psi_L$ are smooth with respect to $t_1,\dots, t_m$ on every coordinate subset on $\cX$.

\item \label{Z2} The Chern curvature of $(L,h_L)$ satisfies
$i\Theta_{(L,h_L)} \ge \delta\cdot \omega$
in the sense of currents on $\cX$ for some positive real number $\delta$.

\item \label{Z3} The multiplier ideal sheaf $\mathcal I(h_{L_t}) = \mathcal{O}_{X_t}$ for each $t\in D$, here $h_{L_t}:= h_L|_{X_t}$.
\item \label{Z4} The K\"ahler manifold $\cX$ contains a Zariski open subset which is Stein.
\end{enumerate}
\medskip

\noindent With these assumptions we set
\begin{align} \label{F}
\mathcal F:=p_*((K_{\cX/D}+L)\otimes \mathcal I(h_L)) = p_*(K_{\cX/D}+L).
\end{align}
\noindent By assumption \ref{Z2} and due to K\"ahler version of Ohsawa-Takegoshi theorem, see \cite{CaoOT}, $\mathcal F$ is indeed a vector bundle and $\mathcal F_t = H^0(X_t, K_{X_t} + L_t)$ for every $t \in D$. There is a Hermitian  metric $h_{\mathcal F}$ on $\mathcal F$ induced by $h_L$, i.e., for any $u_t \in \mathcal F_t$,
$$ \|u_t\|_{h_{\mathcal F}}^2 = \int_{X_t} c_n u\wedge \bar u e^{-\psi_L} ~~\text{with} ~~ c_n=(\sqrt{-1})^{n^2}.$$

\noindent By assumption \ref{Z1}, we know the integrals of the form $\int_{\Omega\cap X_t}|u_t|^2 e^{-\psi_L}$ vary smoothly with $t$ and keep the smoothness under the change of coordinate. The metric $\|\cdot \|_{h_{\mathcal F}}$ on the direct image sheaf $\mathcal F$ is well-defined and smooth.

\begin{rem}
It would be better and more intrinsic to rephrase the assumption \ref{Z1} as follows: Let $V_i$ be the smooth horizontal lift of vector fields $\frac{\partial}{\partial_{t_i}}$, and we assume that $\psi_L$ are smooth concerning these horizontal vector fields. This condition can make the metric $\|\cdot \|_{h_{\mathcal F}}$ on the direct image sheaves smooth.
\end{rem}

\begin{rem}
If $p: \cX \to D$ is the projective morphism, the assumption \ref{Z4} is satisfied.
\end{rem}

\begin{thm} \label{Thm2}
Under the set-up of case $Z$, the Hermitian holomorphic vector bundle $(\mathcal F, h_{\mathcal F})$ over $D$ defined in \eqref{F} is positive in the sense of Nakano.
\end{thm}

\begin{ex} \label{3-ex}
We give some examples that satisfy the assumptions \ref{Z1} and \ref{Z3},see also the Section $2$ in \cite{CGP21}. We assume that there exists a divisor $E= E_1+\dots + E_N$  whose support is contained in the total space $\cX$ of $p$
such that the following requirements are fulfilled. The divisor $E$ intersects each fiber transversally, i.e., for every $t\in D$ the restriction divisor $E_t:= E|_{X_t}$ of $E$ on each fiber $X_t$ has simple normal crossings. Let $\Omega\subset \cX$ be a coordinate subset on $\cX$. We take $(z_1,\dots z_n, t _1,\dots, t_m)$ a coordinate system on $\Omega$ such that the last $m$ variables $t_1,\dots, t_m$ corresponds to the map $p$ itself and such that $z_1\dots z_r= 0$ is the local equation of $E\cap \Omega$.
\begin{enumerate}
\item The metric $h_L$ has \textbf{Poincar\'e type singularities} along $E$, i.e., its local weights $\psi_L$ on $\Omega$ can be written as
\begin{equation*}
\psi_L \equiv - \sum_{I} b_I\log\left(\big(\prod_{i\in I}|z_i|^{2m_i}\big)\big(\phi_I(z)-\log \big(\prod_{i\in I}|z_i|^{2k_i}\big)\big)\right)
\end{equation*}
modulo $\mathcal C^\infty$ functions, where $b_I$ are positive real numbers for all $I$, $m_i, k_i$ are non-negative real number. All $(\phi_I)_I$ are smooth functions on $\Omega$. The set of indexes in the sum coincides with the non-empty subsets of $\{1,\dots, r\}$.
\item  The metric $h_L$ has \textbf{logarithmic type singularities} along $E$, i.e., its local weights $\psi_L$ on $\Omega$ can be written as
\begin{equation*}
\psi_L \equiv - \sum_{I} b_I\log \big(\phi_I(z)-\log (\prod_{i\in I}|z_i|^{2k_i}) \big)
\end{equation*}
modulo $\mathcal C^\infty$ functions, where $b_I$ are positive real numbers satisfying that $b_I < 1$ for all $I$, all $k_i$ are non-negative integers and $(\phi_I)_I$ are smooth functions on $\Omega$. The set of indexes in the sum coincides with the non-empty subsets of $\{1,\dots, r\}$.
\item The metric $h_L$ has \textbf{klt type singularities} along $E$, i.e., its local weights $\psi_L$ on $\Omega$ can be written as
\begin{equation*}
\psi_L \equiv  \sum_{i\in I} a_i\log |z_i|^2
\end{equation*}
modulo $\mathcal C^\infty$ functions, where $a_i$ are real numbers satisfying that $a_i < 1$ for all $i$. The set of indexes in the sum coincides with the non-empty subsets of $\{1,\dots, r\}$.
\end{enumerate}
\end{ex}

In the rest of this section, we will prove Theorem \ref{Thm2}. One of the main results in \cite{DNWZ20} was the following characterization of Nakano positivity in terms of optimal $L^2$-estimate condition.

\begin{thm} \cite[Theorem $1.1$]{DNWZ20} \label{thm:theta-nakano text_intr}
Let $(X,\omega)$ be a  K\"{a}hler manifold of dimension $n$ with a K\" ahler metric $\omega$, and it admits a positive Hermitian holomorphic line bundle, let $(E,h)$ be a smooth Hermitian vector bundle over $X$, and a smooth $\theta\in \mathcal{C}^0(X,\wedge^{1,1}T^*_X\otimes End(E))$ such that $\theta^*=\theta$. If for any $f\in\mathcal{C}^\infty_c(X,\wedge^{n,1}T^*_X\otimes E\otimes A)$ with $\bar\partial f=0$,
and any positive Hermitian line bundle $(A,h_A)$ on $X$ with $i\Theta_{(A,h_A)}\otimes \rm Id_E+\theta>0$ on $\text{supp} f$,
there exist a $u\in L^2(X,\wedge^{n,0}T_X^*\otimes E\otimes A)$, satisfying $\bar\partial u=f$ and
$$\int_X|u|^2_{h\otimes h_A}dV_\omega\leq \int_X\langle B_{(h_A,\theta)}^{-1}f,f\rangle_{h\otimes h_A} dV_\omega,$$
provided that the right-hand side is finite,
where $B_{(h_A,\theta)}=[i\Theta_{(A,h_A)}\otimes \rm Id_E+\theta,\Lambda_\omega]$,
then $i\Theta_{(E,h)}\geq\theta$ in the sense of Nakano.
\end{thm}

\begin{rem}\label{rem:reduce to trivial bundle}
As remark 1.2 in \cite{DNWZ20} said, if $X$ admits a strictly plurisubharmonic function, we can take $A$ to be the trivial bundle (with nontrivial metrics).
\end{rem}

The following lemmas are important for solving the $\dbar$-equation with $L^2$ estimate.

\begin{lem} \cite[Lemma 3.2]{Dem82} \cite[Appendix]{DNWZ20} \label{operator}
Let $X$ be a complex manifold with dimension $n$, assume that $\theta \in \wedge^{1,1}T_X^*$ be a positive $(1,1)$-form, and fix an integer $q\geq 1$.
\begin{enumerate}
\item  for each form $u \in \wedge^{n,q}T_X^*, \langle [\theta, \Lambda_{\omega}]^{-1} u, u \rangle dV_{\omega}$ is non-increasing with respect to $\theta$ and $\omega$;
\item  for each form $u \in \wedge^{n,1}T_X^*, \langle [\theta, \Lambda_{\omega}]^{-1} u, u \rangle dV_{\omega}$ is independent with respect to $\omega$.
\end{enumerate}
\end{lem}

\noindent We need the Richberg-type global regularization result of unbounded quasi-psh functions. Recall an upper semi-continuous function $\phi: X \rightarrow [-\infty, +\infty)$ on a complex manifold $X$ is quasi-psh if it is locally of the form $\phi = u + f$ where $u$ is plurisubharmonic(psh) function and $f$ is a smooth function.

\begin{lem} \cite[Theorem $3.8$]{Bou17} \label{rich}
Let $\phi$ be a quasi-psh function on a complex $X$, and assume given finitely many closed, real $(1,1)$-forms $\theta_{\alpha}$ such that $\theta_{\alpha} + i \ddbar \phi \geq 0$ for all $\alpha$. Suppose either that $X$ is strongly pseudoconvex, or that $\theta_{\alpha}>0$ for all $\alpha$. Then we can find a sequence $\phi_j \in \mathcal{C}^{\infty}(X)$ with the following properties:
\begin{enumerate}
\item  $\phi_j$ converges point-wise to $\phi$;
\item  for each relatively compact open subset $U \Subset X$, there exists $j_U\gg 1$ such that the sequence $(\phi_j)$ becomes decreasing with $\theta_{\alpha} + i \ddbar \phi_j >0$ for each $\alpha$ when $j\geq j_U$.
\end{enumerate}
\end{lem}

\begin{lem}\cite[Theorem 4.5]{Dem12} \label{thm: L2 estimate Nakano}
Let $(X,\omega)$ be a complete K\"ahler manifold, with a K\"ahler metric which is not necessarily complete. Let $(E,h)$ be a Hermitian vector bundle of rank $r$ over $X$, and assume that the curvature operator $B:=[i\Theta_{(E,h)},\Lambda_\omega]$ is semi-positive definite everywhere on $\wedge^{n,q}T_X^*\otimes E$, for some $q\geq 1$. Then for any form $g\in L^2(X,\wedge^{n,q}T^*_{X}\otimes E)$ satisfying $\bar{\partial}g=0$ and $\int_X\langle B^{-1}g,g\rangle dV_\omega<+\infty$, there exists $f\in L^2(X,\wedge^{n,q-1}T^*_X\otimes E)$ such that $\bar{\partial}f=g$ and $$\int_X|f|^2dV_\omega\leq \int_X\langle B^{-1}g,g\rangle dV_\omega.$$
\end{lem}

\begin{thm} \label{equ}
Let $p:\cX\to D$ be a holomorphic proper fibration from a $(n+m)$-dimensional K\"ahler manifold $\cX$ onto the bounded pseudoconvex domain $D\subset \mathbb C^{m}$, and let $(L,h_L=e^{-\psi})$ be a holomorphic line bundle endowed with a possibly singular hermitian metric $h_L$ with local weight $\psi$ and curvature current $\xu \Theta_{(L,h_L)}\geq \delta\cdot p^{\ast}\omega_0$ for the standard K\"ahler form $\omega_0$ on $D$ and some positive number $\delta$. We assume that the K\"ahler manifold $\cX$ contains a Stein Zariski open subset and $\phi$ be any smooth strictly plurisubharmonic function on $D$. 
If
$v \in L^2_{loc}(\cX, \wedge^{n,1} T^{\ast}_{\cX} \otimes L)$
satisfying $\dbar v=0$ and
$$ \int_{\cX} \langle [\delta p^*\omega_0 + i \ddbar p^{\ast}\phi, \Lambda_{\omega}]^{-1} v, v \rangle_{\psi} e^{-p^{\ast}\phi} dV_{\omega} < \infty.
$$
Then $v= \dbar u$ for some $u\in L^2(\cX, \wedge^{n,0} T^{\ast}_{\cX} \otimes L)$ with satisfies
\begin{equation} \label{eq-optimal}
\int_{\cX} |u|^2_{\psi} e^{-p^{\ast}\phi} dV_{\omega} \leq \int_{\cX} \langle [\delta p^*\omega_0 + i \ddbar p^{\ast}\phi, \Lambda_{\omega}]^{-1} v, v \rangle_{\psi} e^{-p^{\ast}\phi} dV_{\omega}.
\end{equation}
Here the subscript $|\cdot|^2_{\psi}$ 
means the inner product with respect to metric weight $\psi$ of $L$.
\end{thm}

\begin{proof}
Firstly, we note $h_L \cdot e^{-p^{\ast}\phi}$ is also the singular metric of $L$ because $p^{\ast}\phi$ be a globally function on $\cX$. Therefore we have $i \Theta_{(L, h_L)} + i \ddbar p^{\ast}\phi \geq \delta p^*\omega_0 + i \ddbar p^{\ast}\phi$ in the sense of currents. To prove the claim, we need an $L^2$-version of the Riemann extension principal. This is to say, if $\alpha\in L^2_{loc}$ be a $L$-valued form on a complex manifold $X$ such that $\dbar \alpha=\beta$ outside a closed analytic subset $A \subset X$, then $\dbar \alpha = \beta$ holds on the whole $X$. On the other hand, if $X$ is a Stein manifold and $L$ be a line bundle on $X$, there exists a hypersurface $H \subset X$ such that $X \setminus H$ is Stein and $L$ is trivial on $X\setminus H$.
Thanks to this, we can assume that $\cX$ is Stein and $L$ is trivial on $\cX$. Then the metric $h_L= e^{-\psi}$ and its local weight $\psi$ is globally defined on $\cX$. Now we can use the global regularization of unbounded quasi-psh. 

By the Lemma \ref{rich}, we may find an exhaustion of $\cX$ by weakly pseudoconvex open subsets $\Omega_j$ such that $\psi_j = \psi|_{\Omega_j}$ is the decreasing limit of sequence $\psi_{j,k} \in \mathcal C^{\infty}(\Omega_j)$ with
$$ i \ddbar \psi_{j,k} \geq \delta p^*\omega_0  \Longrightarrow  i \ddbar \psi_{j,k} + i \ddbar p^{\ast}\phi \geq \delta p^*\omega_0 + i \ddbar p^{\ast}\phi.
$$
Because weakly pseudoconvex manifold admits a complete K\"ahler metric, on $\Omega_j$ we can solve the classical $\dbar$ equation with the $L^2$-estimate as Lemma \ref{thm: L2 estimate Nakano}, i.e., there exist $u_{j,k} \in L^2(\Omega_j, \wedge^{n,0} T^{\ast}_{\Omega_j} \otimes L)$ such that $\dbar u_{j,k} = v$ on $\Omega_j$ and
\begin{align}
\nonumber \int_{\Omega_j} |u_{j,k}|^2_{\psi_{j,k}} e^{-p^{\ast}\phi} dV_{\omega} &=  \int_{\Omega_j} |u_{j,k}|^2 e^{-\psi_{j,k}} e^{-p^{\ast}\phi} dV_{\omega} \\
\nonumber &\leq \int_{\Omega_j} \langle [\delta p^*\omega_0 + i \ddbar p^{\ast}\phi, \Lambda_{\omega}]^{-1} v, v \rangle_{\psi_{j,k}} e^{-p^{\ast}\phi} dV_{\omega} \\
\nonumber&\leq \int_{\Omega_j} \langle [\delta p^*\omega_0 + i \ddbar p^{\ast}\phi, \Lambda_{\omega}]^{-1} v, v \rangle_{\psi_{j}} e^{-p^{\ast}\phi} dV_{\omega} \\
\nonumber&\leq \int_{\cX} \langle [\delta p^*\omega_0 + i \ddbar p^{\ast}\phi, \Lambda_{\omega}]^{-1} v, v \rangle_{\psi} e^{-p^{\ast}\phi} dV_{\omega} \\
\nonumber&= C (\rm constant).
\end{align}
The second inequality because of $\psi_j$ is the decreasing limit of a sequence $\psi_{j,k}$. By monotonicity of $(\psi_{j,k})_k$, we know the integration $\int_{\Omega_j} |u_{j,k}|^2 e^{-p^{\ast}\phi-\psi_{j,l}} dV_{\omega} \leq M$ for $k\geq l$, this shows in particular that $(u_{j,k})_k$ is bounded in $L^2(\Omega_j, e^{-p^{\ast}\phi-\psi_{j,l}})$. After passing to the subsequence, we thus assume that $u_{j,k}$ converges weakly in $L^2(\Omega_j, e^{-p^{\ast}\phi-\psi_{j,l}})$ to $u_j$, which may further be assumed to be the same for all $l$, by a diagonal argument. Now we have $\dbar u_j = v$, and $\int_{\Omega_j} |u_{j}|^2 e^{-p^{\ast}\phi-\psi_{j,l}} dV_{\omega} \leq M$ for all $l$, therefore $\int_{\Omega_j} |u_{j}|^2 e^{-p^{\ast}\phi-\psi} dV_{\omega} \leq M$ by monotone convergence of $\psi_{j,l} \rightarrow \psi$. Once again by a diagonal argument, we may arrange that $u_j \rightarrow u$ weakly in $L^2(K, e^{-p^{\ast}\phi-\psi})$ for each compact subset $K \subset \cX$, and finally we are led to the desired conclusion.
\end{proof}

Here, the operator $[\delta p^*\omega_0 + i \ddbar p^{\ast}\phi, \Lambda_{\omega}]$ is semi-positive on $\cX$ but positive in the $t\in\Omega$ direction, therefore the integral finite condition for $v$ is satisfied if all coefficients of $v$ depends on $t$.
We can now prove Theorem \ref{Thm2} by following Deng--Ning--Wang--Zhou's approach.
\begin{thm}\label{thm: direct image-optimal L2 estimate}
Under the set-up of case $Z$, the Hermitian holomorphic vector bundle $(\mathcal F, \|\cdot\|_{h_{\mathcal F}})$ over $D$ defined in \eqref{F} is positive in the sense of Nakano.
\end{thm}

\begin{proof}
According to Theorem \ref{thm:theta-nakano text_intr}, it suffices to prove that $(\mathcal F = p_*(K_{\cX/D} + L), \|\cdot\|_{h_{\mathcal F}})$
satisfies the optimal $L^2$-estimate condition with the standard K\"ahler metric $\omega_0$ and some $\theta = \delta \cdot\omega_0 \otimes e = \delta \cdot\omega_0$ on $D \subset\mc^n$, here $\delta>0$ and $e= \rm Id_E \in \Gamma(D, End(E))$. Let $\omega$ be a K\"ahler metric on $\cX$.

Let $f$ be a $\bar\partial$-closed compact supported smooth $(m,1)$-form with values in $\mathcal F$, and let $\phi$ be any smooth strictly plurisubharmonic function on $D$. We can write $f(t)=dt\wedge(f_1(t)d\bar t_1+\cdots +f_n(t)d\bar t_n)$, with each $f_i(t)\in \mathcal F_t=H^0(X_t, K_{X_t}\otimes L)$. One can identify $f$ as a smooth compact supported $(n+m,1)$-form $\tilde f(t,z):=dt \wedge (f_1(t,z)d\bar t_1+\cdots+f_n(t,z)d\bar t_n)$ on $\cX$,
with $f_i(t,z)$ being holomorphic section of $K_{X_t}\otimes L|_{X_t}$.
We have two observations as follows, the first is that $\dbar_z f_i(t,z)=0$ for any fixed $t\in D$, since  $f_i(t,z)$ are holomorphic sections $K_{X_t}\otimes L|_{X_t}$. The second is that $\dbar_tf=0$, since $f$ is a $\bar\partial$-closed form on $D$.
It follows that $\tilde f$ is a $\bar\partial$-closed compact supported  $(n+m,1)$-form on $\cX$ with values in $L$.
We want to solve the equation $\bar\partial\tilde u=\tilde f$ on $X$ by using Theorem \ref{equ}.
Now we equipped $L$ with the metric $\tilde h:=he^{-p^*\phi}$,
then $i\Theta_{(L,\tilde h)}=i\Theta_{(L,h)} + i\ddbar p^*\phi \geq \delta p^*\omega_0 + i\ddbar\ p^*\phi$, which is also positive in the sense of currents. Hence there is  $\tilde u\in \wedge^{m+n,0}T^{\ast}_{\cX}\otimes L$, such that $\bar\partial\tilde u=\tilde f$, and satisfies the following estimate
\begin{align}\label{eqn: optimal L2 estimate 1}
\int_{\cX} c_{m+n}\tilde u\wedge \bar{\tilde u}e^{-\psi -p^*\phi} =& \int_{\cX} |\tilde u|^2 e^{-\psi-p^*\phi}dV_{\omega}  \notag\\
\leq &\int_{\cX}\langle [\delta p^*\omega_0 +i\partial\bar\partial p^*\phi, \Lambda_{\omega}]^{-1}\tilde f,\tilde f\rangle e^{-\psi-p^*\phi}dV_\omega \notag\\
=& \int_{\cX}\langle [\delta p^*\omega_0 + i\partial\bar\partial p^*\phi, \Lambda_{\omega'}]^{-1}\tilde f,\tilde f\rangle e^{-\psi-p^*\phi}dV_{\omega'} \notag\\
=&\int_D \langle [ \delta \omega_0 + i\partial\bar\partial \phi, \Lambda_{\omega_0}]^{-1}f,f\rangle_te^{-\phi}dV_{\omega_0}.
\end{align}
The first inequality due to \eqref{eq-optimal}, the second equality holds because $\tilde f$ is $(n+m,1)$-form, and therefore $\langle [\delta p^*\omega_0 + i\partial\bar\partial p^*\phi, \Lambda_{\omega}]^{-1}\tilde f,\tilde f\rangle dV_\omega$ are independent to $\omega$ in view of Lemma \ref{operator}. The last equality is valid because here we choose $\omega'= i \Sigma_{j=1}^m dt_j \wedge d\bar t_j + i \Sigma_{j=1}^n dz_j \wedge d\bar z_j$. The notation $\langle\cdot,\cdot\rangle_t$ here means a pointwise inner product concerning the Hermitian metric $\|\cdot\|_{h_{\mathcal F}}$ on $\mathcal F$.

Set $\tilde u_t:=\tilde u(t,\cdot)$, we observe that $\dbar \tilde u_t=0$ for any fixed $t\in D$, since $\bar\partial\tilde u=\tilde f$ and the $(n+m,1)$-form $\tilde f$ contains only the terms of $d\bar t_i$.
This means that $\tilde u_t \in \mathcal F_t$, and hence we may view $\tilde u$ as a section $u$ of $\mathcal F$.
It is that $\bar\partial u=f$.
Due to Fubini's theorem, we have that
\begin{align}\label{eqn: optimal L2 estimate 2}
\int_{\cX} c_{m+n}\tilde u\wedge \bar{\tilde u}e^{-\psi-p^*\phi} = \int_D \|u_t\|_{h_{\mathcal F}}^2e^{-\phi}dV_{\omega_0}.
\end{align}
Combining \eqref{eqn: optimal L2 estimate 1} with \eqref{eqn: optimal L2 estimate 2}, we obtain
\begin{align*}
\int_D \|u_t\|_{h_{\mathcal F}}^2e^{-\phi} dV_{\omega_0} \leq \int_D \langle [ \delta \omega_0 + i\partial\bar\partial \phi, \Lambda_{\omega_0}]^{-1}f,f\rangle_t e^{-\phi}dV_{\omega_0}.
\end{align*}
So we have proved that $\mathcal F$ satisfies the optimal $L^2$-estimate condition, thus owing to Theorem \ref{thm:theta-nakano text_intr} (and Remark \ref{rem:reduce to trivial bundle}), one get $\xu\Theta_{(\mathcal F, h_{\mathcal F})} \geq \delta\cdot \omega_0$. This is to say, $(\mathcal F, h_{\mathcal F})$ is positive in the sense of Nakano.
\end{proof}

Finally, we prove Corollary \ref{ko-oh} using the following remark.

\begin{rem}[Remark on the Nadel--Nakano vanishing theorem]
In general, the Nakano positivity of a vector bundle implies the vanishing of corresponding cohomologies, even in the presence of a singular metric. The induced singular Hermitian metric on the direct image of the adjoint nef and big line bundle gives a concrete example of the Nakano positivity of singular Hermitian metrics in the sense of \cite{Iwa21}.
In M. Iwai's paper \cite[Theorem $1.2$, Theorem $1.3$]{Iwa21}, the author proposes the following three conditions which imply coherence and vanishing. Let $(E, h)$ be a holomorphic vector bundle on $X$ with a singular Hermitian metric.
\begin{enumerate}
\item \label{11} There exists a proper analytic subset $Z$ such that $h$ is smooth on $X\backslash Z$.
\item \label{22} the Hermitain metric $he^{-\zeta}$ on $E$ is a Griffiths semi-positive for some continuous function $\zeta$ on $X$.
\item \label{33} There exists a real number $C$ such that $\xu\Theta_{E,h}- C\omega\otimes \rm Id_E\geq0$ on $X\backslash Z$ in the sense of Nakano.
\end{enumerate}
When the $C$ in the above condition \eqref{33} is just a real number, then the corresponding sheaf $\mathcal{E}(E,h)$ is coherent. If moreover, the $C$ is a positive real number, then we have the following vanishing:
$$
H^q(X, K_X\otimes\mathcal{E}(E,h)) =0
$$
for any $q\geq 1$. Here $\mathcal{E}(E, h)$ is the $L^2$-subsheaf defined by $\mathcal{E}(E, h)_x:= \{s_x \in \mathcal{O}(E)_x : |s|^2_{h}\,\, \text{is locally integrable around x} \}$.
\end{rem}

\begin{proof}[Proof of Corollary \ref{ko-oh}]
    From Theorem \ref{semi-po} and Lemma \ref{loc}, the holomorphic vector bundle $(\mathcal{F}, h_{\mathcal{F}})$ satisfies the aforementioned three conditions. 
    Then we have that 
    \begin{align*}
        H^q(Y,K_Y\otimes\mathcal{E}(\mathcal{F},h_{\mathcal{F}}))=0
    \end{align*}
    for any $q>0$. 
    
    Thus, it suffices that we show that $\mathcal{E}(\mathcal{F},h_{\mathcal{F}})=\mathcal{O}_Y(\mathcal{F})$.
    Here, we may assume that $h_L=e^{-\phi_L}$ can be taken to satisfy the Lelong number $\nu(\phi_L,x)<1$ for any $x\in X$ from nef and big (see \cite[Corollary\,6.19]{Dem10}).
    For any $t\in Y$ and any $s\in \mathcal{F}_t$, we take a small neighborhood $U$ of $t$ satisfying that $K_Y|_U$ is trivial. 
    The Lebesgue measure of a set $Z_f:=\{t\in Y\mid t\,\, \text{is not a regular point of}\,\,f\}$ 
    is zero and $f$ is submersion on $Y\setminus Z_f$.
    By the relation $H^0(B,\mathcal{F})\cong H^0(B,K_Y\otimes\mathcal{F})=H^0(f^{-1}(U),K_X\otimes L):s=s\,dt\mapsto \widetilde{s}$ and Fubini's theorem, we have that 
    \begin{align*}
        \int_U||s||^2_{h_{\mathcal{F}}}dV_U=\int_{U\setminus Z_f}||s||^2_{h_{\mathcal{F}}}dV_U=\int_{f^{-1}(U\setminus Z_f)}c_{m+n}\widetilde{s}\wedge\overline{\widetilde{s}} e^{-\phi_L},
    \end{align*}
    where set $\widetilde{s}_t:=\widetilde{s}(t,\cdot)$ then $\widetilde{s}_t\in\mathcal{F}_t$.

    From Skoda's results (see \cite{Sko72}) and Lelong number conditions, the integral of $e^{-\phi_L}$ on a small neighborhood of any point in $X$ is finite.
    By compact-ness of $X_t$ and holomorphicity of $s_t$, we get $\max_{x\in X_t}|s_t(x)|^2<+\infty$ and $\sup_{x\in f^{-1}(U)}c_{m+n}\widetilde{s}\wedge\overline{\widetilde{s}}<+\infty$.
    Hance, we have that the integrable of $||s||^2_{h_{\mathcal{F}}}$ on $f^{-1}(U)$ is finite. 
\end{proof}


\begin{thebibliography}{GP}
\bibitem[BKK+15]{BKK+15}
{\scshape T.~Bauer, S.~J Kov\'acs, A.~K\"uronya, E. C.~Mistretta, T.~Szemberg, S.~Urbinati} -- {On postivitity and base loci of vector bundles}, \emph{{Eur. J. Math.}} \textbf{1} (2015), no.2, ~229--249.


\bibitem[Ber09]{Bo09}
{\scshape B.~Berndtsson} -- {Curvature of vector bundles associated to
  holomorphic fibrations}, \emph{{Ann. of Math.}} \textbf{169} (2009),
  ~531--560.

\bibitem[BP09]{BP09}
{\scshape B.~Berndtsson, M. Paun} -- {Bergman kernels and the pseudoeffectivity of relative canonical bundles}, \emph{Duke. Math. J.} \textbf{145} (2008), 341-378.
  p.~531--560.




\bibitem[BLNN23]{BLNN23}
{\scshape I.Biswas, F. Laytimi, D.S. Nagaraj, W. Nahm}--{On the direct image of the adjoint line bundle}, Preprint
  \href{http://arxiv.org/abs/2310.02764}{arXiv: 2310.02764}, 2023.

\bibitem[Bou02]{Bou02}
{\scshape S. Boucksom}--{On the volume of a line bundle},
\emph{Int. J. Math.} 13, no. 10, 1043-1063 (2002).

\bibitem[Bou17]{Bou17}
{\scshape S.~Boucksom}--{Singularities of plurisubharmonic functions and multiplier ideals}, Lecture notes, from the author's website: http://sebastien.boucksom.perso.math.cnrs.fr/notes/L2.pdf.

\bibitem[Cao17]{CaoOT}
{\scshape J.~Cao} -- { Ohsawa-{T}akegoshi extension theorem for compact
  {K}\"{a}hler manifolds and applications}, in \emph{Complex and symplectic
  geometry}, Springer INdAM Ser., vol.~21, Springer, Cham, 2017, p.~19--38.


\bibitem[CGP21]{CGP21}
{\scshape J. Cao, H. Guenancia, M. P{\u{a}}un} --  {Curvature formula for direct images of twisted twisted relative canonical bundles endowed with a singular metric}, \emph{Ann. Fac. Sci. Toulouse Math.}  31 (2022), $n^{\circ}3$, 861-905.


\bibitem[Dem82]{Dem82}
{\scshape J.~P. Demailly} -- {Estimations $L^2$ pour l'op\'erateur $\dbar$ d'un fibr\'e vectoriel holomorphe semi-positif au-dessus d'une variete k\"ahlerienne complete}, \emph{Ann. Sci. \'Ecole Norm. Sup.} \textbf{4} (1982), 15(3), p.~457--511.


\bibitem[Dem12]{Dem12}
{\scshape J.-P. Demailly} -- {Complex {A}nalytic and {D}ifferential
  {G}eometry}, September 2012, Open-Content Book, freely available from the
  author'swebsitee.

\bibitem[Dem10]{Dem10}
{\scshape J.-P. Demailly} --
{Analytic Methods in Algebraic Geometry}, Higher Education Press, Surveys of Modern Mathematics, Vol. 1, 2010



\bibitem[DNWZ22]{DNWZ20}
{\scshape F.~Deng, J.~Ning, Z.~Wang  X.~Zhou} -- {
  Positivity of holomorphic vector bundles in terms of $L^p$-conditions of
  $\bar\partial$}, Math. Ann. (2022). https://doi.org/10.1007/s00208-021-02348-7.
  





\bibitem[HPS18]{HPS18}
{\scshape C.~Hacon  M.~Popa  C,~Schnell} -- {Algebraic fiber spaces over abelian varieties: Around a recent theorem by Cao and P{\u{a}}un}, \emph{Local and global methods in algebraic geometry, comtemp. Math} \textbf{712} (2018),
  p.~143--195.

 \bibitem[Iwa21]{Iwa21}
 {\scshape M. Iwai}-- \emph{Nadel-Nakano vanishing theorems of vector bundles with singular Hermitian metrics}, \emph{ Annales de la Facult\'e des sciences de Toulouse: Math\'ematiques} (6), \textbf{30} (2021), no.~1, 63--81. 


\bibitem[MM07]{MM07}
{\scshape X. Ma, G. Marinescu}, \textit{Holomorphic Morse inequalities and Bergman kernels,}{\it Progress in Mathematics 254,} Birkhäuser Boston, Inc., Boston, MA. 2007, 422 pp.





\bibitem[Mou97]{Mou97}
{\scshape C. Mourougane}, \textit{Image direct de fibr\'e en droites adjoints}, {\it Proc. Res. Ins. Math. Sci.} {\bf 33} (1997), 893--916.

\bibitem[MT07]{MT07}
{\scshape C. Mourougane, S. Takayama}, \textit{Hodge metrics and positivity of direct images} 
\emph{J. reine angew. Math} {\bf 606} (2007), 167--178.





\bibitem[PT18]{PT}
{\scshape M.~P\u{a}un, S.~Takayama} -- { Positivity
  of twisted relative pluricanonical bundles and their direct images},
  \emph{J. Algebraic Geom.} \textbf{27} (2018), no.~2, p.~211--272.

\bibitem[Rau15]{Rau15}
{\scshape H.~Raufi} -- {Singular hermitian metrics on holomorphic vector bundle}, \emph{Ark. Math.} \textbf{53} (2015), no.~2, p.~359--382.



\bibitem[Sko72]{Sko72}
    {\scshape H. Skoda} -- {Sous-ensembles analytiques d'ordre fini ou infini dans $\mathbf{C}^n$}, Bull. Soc. Math. France \textbf{100} (1972), 353-408.


%

\end{thebibliography}
\end{document}